\documentclass[reqno]{amsart}
\usepackage{amsfonts,amsthm,amsmath,amssymb}
\usepackage{graphicx}
\usepackage{tabularx}
\usepackage{array}
\usepackage[usenames,dvipsnames]{color}
\usepackage{comment}
\usepackage{fullpage}
\usepackage{xcolor}
\usepackage{tikz}

\tikzstyle{legend_general}=[rectangle, rounded corners, thin,
                          top color= white,bottom color=lavander!25,
                          minimum width=2.5cm, minimum height=0.8cm,
                          violet]

\DeclareMathOperator{\ch}{ch}
\DeclareMathOperator{\AT}{AT}

\DeclareMathOperator{\supp}{Supp}
\DeclareMathOperator{\coeff}{coeff}
\DeclareMathOperator{\perm}{per}

\newtheorem{theorem}{Theorem}[section]

\newtheorem{conjecture}[theorem]{Conjecture}
\newtheorem{question}[theorem]{Question}

\newtheorem{lemma}[theorem]{Lemma}

\theoremstyle{definition}

\title{List-avoiding orientations}

\author{Peter Bradshaw$^*$}
\address{Department of Mathematics, University of Illinois Urbana-Champaign}
\email{pb38@illinois.edu}
\thanks{$^*$This project was partially funded by NSF RTG grant DMS-1937241}

\author{Yaobin Chen}
\address{Shanghai Center for Mathematical Sciences, Fudan University, Shanghai, 200438, China}
\email{ybchen21@m.fudan.edu.cn}

\author{Hao Ma}
\address{Faculty of General Education and Foreign Languages, Anhui Institute of Information Technology, Wuhu, 241199, China}
\email{haoma5@iflytek.com}

\author{Bojan Mohar$^{\dagger}$}
\address{Department of Mathematics, Simon Fraser University, Burnaby, BC, Canada}
\email{mohar@sfu.ca}
\thanks{$^{\dagger}$Supported in part by the NSERC Discovery Grant R611450 (Canada),
and by the Research Project N1-0218 of ARRS (Slovenia). On leave from IMFM, Department of Mathematics, University of Ljubljana.}

\author{Hehui Wu$^{\dagger \dagger}$}
\address{Shanghai Center for Mathematical Sciences, Fudan University, Shanghai, 200438, China}
\email{hhwu@fudan.edu.cn}
\thanks{$^{\dagger \dagger}$Supported in part by National Natural Science Foundation of China grant 11931006, National Key Research and Development Program of China (Grant No. 2020YFA0713200), and the Shanghai Dawn Scholar Program grant 19SG01.}

\begin{document}
\maketitle
\begin{abstract}
Given a graph $G$ with a set $F(v)$ of forbidden values at each $v \in V(G)$,
an $F$-avoiding orientation of $G$ is an orientation in which $\deg^+(v) \not \in F(v)$ for each vertex $v$.
Akbari, Dalirrooyfard, Ehsani, Ozeki, and Sherkati 
conjectured that if $|F(v)| < \frac{1}{2} \deg(v)$ for each $v \in V(G)$, then $G$
has an $F$-avoiding orientation, and they showed that this statement is true when $\frac{1}{2}$ is replaced by $\frac{1}{4}$.
In this paper, we take a step toward this conjecture by proving that
if $|F(v)| < \lfloor \frac{1}{3} \deg(v) \rfloor$ for each vertex $v$, then $G$ has an $F$-avoiding orientation. Furthermore, we show that if the maximum degree of $G$ is subexponential in terms of the minimum degree, then this coefficient of $\frac{1}{3}$ can be increased to $\sqrt{2} - 1 - o(1) \approx 0.414$.
Our main tool is a new sufficient condition for the existence of an $F$-avoiding orientation based on the Combinatorial Nullstellensatz of Alon and Tarsi.
\end{abstract}

\section{Introduction}

An \emph{orientation} of a graph $G$ is an assignment of a direction $uv$ or $vu$ to each edge $\{u,v\} \in E(G)$. For an orientation $D$ of a graph $G$ and a vertex $v \in V(G)$, we denote by 
$E^+(v)$ the arcs outgoing from $v$ in $D$, and we denote by $E^-(v)$ the arcs incoming to $v$. We write $\deg_D^+(v) = |E^+(v)|$ and $\deg_D^-(v) = |E^-(v)|$.
We write $\mathbb N$ for the set $\{0,1, 2,\dots\}$ of nonnegative integers.
In 1976, Frank and Gy{\'a}rf{\'a}s~\cite{gyarfas1978orient} proved that for a graph $G$ and two mappings $a,b:V(G)\xrightarrow{}\mathbb{N}$ satisfying $a(v)\leq b(v)$ for every vertex $v$, $G$ has an orientation $D$ satisfying $a(v)\leq \deg_D^+(v)\leq b(v)$ for every vertex $v$ if and only if for each subset $U\subseteq V(G)$,
\[\sum\limits_{v\in U}a(v)-d(U) \leq |E(G[U])|\leq \sum\limits_{v\in U}b(v),\]
where $d(U)$ is the number of edges joining $U$ and $\overline{U} = V(G)\setminus U$, and $G[U]$ is the subgraph of $G$ induced by $U$.

Recently, Akbari, Dalirrooyfard, Ehsani, Ozeki and Sherkati~\cite{akbari2020orientations}
considered the similar problem of finding an orientation of a graph that avoids a certain out-degree at each vertex. 
Given a graph $G$ and a function $f:V(G)\xrightarrow{}\mathbb{N}$, we
say that an orientation $D$ of $G$ is \emph{$f$-avoiding} if  $\deg_D^+(v)\neq f(v)$ for each $v\in V(G)$.
It was proved in \cite{akbari2020orientations} that there is an $f$-avoiding orientation for every 2-connected graph $G$ that is not an odd cycle and for
every function $f:V(G)\xrightarrow{}\mathbb N$, and that an odd cycle has an $f$-avoiding orientation if and only if $f(v)\neq 1$ for some vertex $v$ of the cycle.
Frank, Tardos, and Seb{\H{o}}~\cite{frank1984covering} 
considered the same problem modulo $2$, or in other words, with parity constraints on the out-degrees of the vertices in $G$.

In addition to considering orientations that avoid a single given out-degree at each vertex, Akbari et al.~considered the graph orientation problem in which each vertex has a list of forbidden out-degrees.
Given a graph $G$ and a function $F:V(G)\xrightarrow{}2^{\mathbb{N}}$, an orientation $D$ of $G$ is said to be \emph{$F$-avoiding} if $\deg_D^+(v)\notin F(v)$ for each $v\in V(G)$.

The problem of finding an $F$-avoiding orientation of a graph $G$ can be represented in the language of \emph{general factors},
introduced by Lov\'asz \cite{LovaszFactor} and defined as follows.
Given a graph $G$ and a function $H:V(G) \rightarrow 2^{\mathbb N}$, a spanning subgraph $G' \subseteq G$ is an \emph{$H$-factor} if $\deg_{G'}(v) \in H(v)$ for each $v \in V(G)$. 
To represent the problem of finding an $F$-avoiding orientation of $G$ in the setting of $H$-factors,
we construct a graph $G^{(1)}$ 
by replacing each edge $e \in E(G)$ with a vertex $v_e$ of degree $2$ whose neighbors are the endpoints of $e$, and we define a function $H:V(G^{(1)}) \rightarrow 2^{\mathbb N}$ so that $H(v) = \mathbb N \setminus F(v)$ for each $v \in V(G)$ and $H(v_e) = \{1\}$ for each $e \in E(G)$. 
Then, it is easy to check that the problem of finding an $F$-avoiding orientation of $G$ is equivalent to the problem of finding an $H$-factor of $G^{(1)}$.
Cornu\'ejols \cite{Corn} and Seb\H{o} \cite{Sebo} give polynomial-time algorithms for checking whether
$G$ has an $H$-factor whenever $i \not \in H(v)$ implies that $i+1 \in H(v)$ for each $i \in \mathbb N$ and $v \in V(G)$,
which allow us to check in polynomial time whether $G$ has an $F$-avoiding orientation whenever the function $H$ defined from $F$ as described above satisfies this condition.

The concept of an $F$-avoiding orientation is also related to nowhere-zero flows and in particular to Tutte's $3$-flow conjecture,
which is stated as follows. Given a directed graph $D$ and a positive integer $k$, a \emph{nowhere-zero $k$-flow} on $D$ is an assignment $\phi:E(D) \rightarrow \mathbb Z_k \setminus \{0\}$ satisfying $\sum_{e \in E^+(v)} \phi(v) \equiv \sum_{e \in E^-(v)} \phi(v) \pmod k$ at each vertex $v \in V(G)$.  
When $k = 3$, a nowhere-zero $3$-flow can be interpreted as an orientation of a graph $G$, and an orientation $D$ of $G$ is a nowhere-zero $3$-flow if and only if $\deg^+_D(v) \equiv \deg^-_D(v) \pmod 3$ holds for each vertex $v \in V(G)$.
Tutte's $3$-flow conjecture states 
that every graph $G$ with no edge-cut of size $1$ or $3$ admits a nowhere-zero 3-flow (see \cite[Conjecture $(C_2)$]{Jaeger} or \cite{BondyMurty}). 
It has long been known that Tutte's $3$-flow conjecture is
equivalent to the statement that every 5-regular graph with no edge-cut of size $1$ or $3$
has an $F$-avoiding orientation where $F (v) =\{0,2,3,5\}$ at each vertex $v$  (see  \cite[Unsolved Problem 48]{BondyMurty} for this equivalent formulation).

It is conjectured that for every graph $G$, as long as $F$ forbids less than roughly half of the possible out-degrees at each vertex, then $G$ has an $F$-avoiding orientation.

\begin{conjecture}[\cite{akbari2020orientations}]\label{conj:1}
Let $G$ be a graph, and let $F:V(G)\xrightarrow{}2^{\mathbb{N}}$. If
\[|F(v)|\leq \frac{1}{2} (\deg_G(v)-1) \]
for each $v\in V(G)$, then $G$ has an $F$-avoiding orientation. 
\end{conjecture}

If Conjecture \ref{conj:1} is true, then the upper bound $\frac{1}{2}(\deg_G(v)-1)$ is sharp. 
To show sharpness, take a $2k$-regular graph $G$ on $n$ vertices 
with independence number less than $\frac{n}{k+1}$. Then, $G$
has no $F$-avoiding orientation with $F(v) =  \{k,k+1,\dots,2k-1\}$
for each vertex $v$.
The graph $K_{2k + 1}$ is such an example.

One can also compare the size of each forbidden list $F(v)$ to the out-degree of $v$ in some fixed orientation. Since every graph $G$ has an orientation $D$ in which each $v \in V(G)$ satisfies $\deg_D^+(v) \geq  \lfloor \deg_G(v) / 2 \rfloor$, the following conjecture implies Conjecture \ref{conj:1} with an error of at most $1$.

\begin{conjecture}[\cite{akbari2020orientations}]\label{conj:2}
Let $G$ be a graph, and let $F:V(G)\xrightarrow{}2^{\mathbb{N}}$. If $G$ has an orientation $D$ such that 
$$\deg^+_D(v)\ge |F(v)|+1$$
for every $v\in V(G)$, then $G$ has an $F$-avoiding orientation.
\end{conjecture}

One tool that was used extensively in \cite{akbari2020orientations} is the following theorem, called the Combinatorial Nullstellensatz, introduced by Alon and Tarsi in \cite{AT} and further developed as a tool by Alon \cite{alon_1999}. 

\begin{theorem}[Combinatorial Nullstellensatz, \cite{AT}]
\label{thm:CNSS}
Let $K$ be a field, and let $f$ be a polynomial in the ring $K[x_1, \dots, x_n]$. Suppose that the degree of $f$ is $t_1 + \dots + t_n$, 
where each $t_i$ is a nonnegative integer, and suppose that the coefficient of\/ $\prod_{i = 1}^n x_i^{t_i}$ in $f$ is nonzero. Then, if $S_1, \dots, S_n$ are subsets of $K$ 
each
satisfying $|S_i| > t_i$, then there exist elements $s_1 \in S_1, \dots, s_n \in S_n$ so that 
\[f(s_1, \dots, s_n) \neq 0.\]
\end{theorem}

We note that Shirazi and Verstra{\"e}te~\cite{shirazi2008note} used the Combinatorial Nullstellensatz to prove 
a result similar to Conjecture \ref{conj:1} for $F$-avoiding subgraphs. They 
proved that for every graph $G$, if $F:V(G) \rightarrow 2^{\mathbb N}$ satisfies $|F(v)| < \lfloor \deg_G(v)/2 \rfloor$ for each $v \in V(G)$, then $G$ has a spanning subgraph $H$ such that $\deg_H(v) \not \in F(v)$ for each $v \in V(G)$. Frank, Lau, and Szab\'o \cite{FLS} proved the same result using elementary methods.

In graph theory, the Combinatorial Nullstellensatz is often used for list coloring problems (see e.g.~\cite{AT,Ellingham,Huang} or the monograph \cite{ZhuBook}).
In this case, the polynomial $f$ in the statement of the Combinatorial Nullstellensatz is the \emph{graph polynomial} of a graph $G$,
introduced by Li and Li \cite{LiLi} and defined as
follows. For an arbitrary fixed orientation $D$ of $G$, the \emph{incidence matrix} of $D$ is defined as the matrix $M = (m_{ve}: v \in V(G), e \in E(G))$
for which $m_{ve} = 1$ if $e \in E_G^+(v)$, $m_{ve} = -1$ if $e \in E_G^-(v)$, and $m_{ve} = 0$ otherwise (see e.g.~\cite[Section 2.6]{Gross}). Then, given a field $K$, the \emph{graph polynomial} of $G$ is the polynomial in the ring $K[x_v: v \in V(G)]$
defined as $f^*_D = \prod_{e \in E(G)} \left ( \sum _{v \in V(G)} m_{ve} x_v \right )$. 
Equivalently, $f^*_D = \prod (x_u - x_v)$, where the product is taken over all directed arcs $uv \in E(D)$. With this definition, a proper coloring $V(G) \rightarrow K$ can be translated into an assignment of values of $K$ to the variables $x_v$, and $f^*_D$ takes a nonzero value for a given assignment if and only if the corresponding coloring of $G$ is proper.

One convenient property of the graph polynomial is that its coefficients can be determined solely by counting Eulerian orientations of the graph (see {\cite{alon_1999}} for a complete explanation). These are defined as follows. 
An orientation $D$ of a graph $G$ is said to be \emph{Eulerian} if $\deg_D^+(v)=\deg_D^-(v)$ holds for all $v\in V(G)$. A subgraph $H$ of $G$ is called \emph{even} if $|E(H)|$ is even and is called \emph{odd} otherwise. Given an orientation $D$ of $G$, we let $EE(D)$ and $EO(D)$, respectively, denote the number of even and odd subgraphs of $G$ that are Eulerian with respect to $D$.
If $G$ has an orientation $D$ satisfying $EE(D) \neq EO(D)$, then we say that $D$ is an \emph{Alon-Tarsi orientation}. Alon and Tarsi \cite{AT}
first considered these orientations and
proved the following groundbreaking result, which has deep applications in list-coloring: If $D$ is an Alon-Tarsi orientation of $G$, and if $L$ is a list assignment on $G$
for which $|L(v)| > \deg_D^+(v)$ at each vertex $v \in V(G)$, then $G$ is $L$-choosable.
Motivated by this beautiful result, Jensen and Toft \cite{jensen2011graph}
defined the \emph{Alon-Tarsi number} of a graph $G$, denoted by $\AT(G)$, as the minimum value $k$ such that $G$ has an Alon-Tarsi orientation of maximum out-degree less than $k$. Alon and Tarsi's result can then be rephrased as $\ch(G) \leq \AT(G)$, where $\ch(G)$ is the list chromatic number of the graph.

In their work on $F$-avoiding orientations, Akbari et al.~\cite{akbari2020orientations} also 
apply the Combinatorial Nullstellensatz. However, they do not apply this tool to the graph polynomial defined above. Rather, given a graph $G$,
they define a polynomial $P$ in a ring with a set $\{y_e:e \in E(G)\}$ of algebraically independent variables which can take the values $1$ or $-1$.
It turns out that there is a bijection between variable assignments on $P$ and orientations of $G$. Similarly to the graph polynomial, the coefficients of $P$ can be determined by counting Eulerian orientations in $G$, 
and the following theorem relating $F$-avoiding orientations and Alon-Tarsi orientations holds. 
In fact, the theorem stated in \cite{akbari2020orientations} is slightly weaker, but the following version is an easily proven corollary. 

\begin{theorem}[\cite{akbari2020orientations}]\label{thm:2}
Let $G$ be a graph, let $H$ be a spanning subgraph of $G$, and let $F:V(G)\xrightarrow{}2^{\mathbb{N}}$ be a map. If there exists an Alon-Tarsi orientation $D$ of $H$ such that $|F(v)|\leq \deg_D^+(v)$ for every vertex $v\in V(G)$, then $G$ has an  $F$-avoiding orientation.
\end{theorem}

Theorem \ref{thm:2} shows a deep connection between the list coloring and $F$-avoiding orientation problems.
We will in fact see in Theorem \ref{thm:duals} that this connection is a special case of a more general duality between \emph{dual polynomials} obtained from a common matrix.
In the case of Theorem \ref{thm:2}, the polynomials for list coloring and $F$-avoiding orientations are related via the incidence matrix of a graph.
Furthermore, by using Theorem \ref{thm:2} 
along with a result of Huang, Wong, and Zhu \cite{Huang},
one can immediately prove Conjecture \ref{conj:1} 
for the complete graph $K_n$. Indeed, 
Huang et al.~\cite{Huang} proved that if $N$ is a maximum matching in $K_n$, 
then $\AT(K_n - N) = \lceil n/2 \rceil$,
which implies that $K_n - N$ has an Alon-Tarsi orientation in which each vertex has its out-degree at least 
$( n-2 ) - \lceil n/2 \rceil + 1  = \left \lfloor \frac{1}{2}n - 1 \right \rfloor $.
Hence, Theorem \ref{thm:2} implies that $K_n$ has an $F$-avoiding orientation whenever $|F(v)| \leq \frac{1}{2}n- 1$ for each vertex $v$.

Since every graph $G$ has an orientation $D$ in which $\deg_D^+(v) \geq \left \lfloor \frac{1}{2} \deg_G(v) \right \rfloor$ for each vertex $v$, and since $EE(D) > EO(D) = 0$ whenever $G$ is bipartite, Theorem \ref{thm:2} also implies that a bipartite graph $G$ has an $F$-avoiding orientation whenever $|F(v)| \leq \frac{1}{2}\deg_G(v) $ for each $v$. 
Furthermore, it is well known that every graph $G$ contains a spanning bipartite subgraph $H$ satisfying $\deg_H(v)\geq \frac{1}{2} \deg_G(v)$ for each vertex $v$. 
As a consequence, we have the following result.
\begin{theorem}[\cite{akbari2020orientations}]
\label{thm:fourth}
Let $G$ be a graph, and let $F:V(G)\xrightarrow{}2^{\mathbb{N}}$ be a map. If
$|F(v)|\leq \frac{1}{4}\deg_G(v)$ 
for each vertex $v\in V(G)$, then $G$ has an $F$-avoiding orientation.
\end{theorem}
%
Using a recent result of Lov\'asz, Thomassen, Wu, and Zhang \cite{LovaszThomassen} on Tutte's $3$-flow conjecture,
Akbari et al.~\cite{akbari2020orientations}
also proved that 
when $G$ is $d$-regular and  $d$-edge-connected, it is enough to require that $|F(v)|\leq (d-5)/3$ for each vertex $v \in V(G)$ in order to guarantee an $F$-avoiding orientation. 

Let us now state the main results of this paper.
First, we introduce the following theorem, which gives a factor-$\frac{2}{3}$ approximation towards Conjecture \ref{conj:2}.

\begin{theorem}\label{thm:two_thirds}
Let $G$ be a graph, and let $F:V(G) \rightarrow 2^{\mathbb N}$. If $G$ has an orientation $D$ such that 
\[ |F(v)|\le \frac{2}{3}\deg^+_D(v)-1\]
for each $v \in V(G)$, then $G$ has an $F$-avoiding orientation.
\end{theorem}

Since every graph has an orientation $D$ in which each $v \in V(G)$ satisfies $\deg_D^+(v) \geq  \left \lfloor \frac{1}{2} \deg_G(v) \right \rfloor$, Theorem \ref{thm:two_thirds} also 
implies that Conjecture \ref{conj:1} holds with a coefficient of roughly $\frac{1}{3}$, which improves Theorem \ref{thm:fourth} (see our Theorem \ref{thm:third}). We also show that if $G$ is a graph whose maximum degree is subexponential in terms of its minimum degree, then Conjecture \ref{conj:1} holds with an even greater coefficient of roughly $\sqrt 2 - 1 \approx 0.414$. 

\begin{theorem}
\label{thm:reg_intro}
Let $G$ be a graph of minimum degree $\delta$ and maximum degree $\Delta = e^{o(\delta)}$, and 
let $F:V(G) \rightarrow 2^{\mathbb N}$.
If \[|F(v) |\leq \left (\sqrt 2 - 1 - o(1) \right ) \deg_G(v)\]
for each $v \in V(G)$, then $G$ has an $F$-avoiding orientation.
\end{theorem}

Note that in the statement of Theorem \ref{thm:reg_intro}, our function $o(1)$ approaches $0$ as $\delta \rightarrow \infty$, and the rate at which $o(1)$ approaches $0$ depends on how quickly $\delta$ dominates $\log \Delta$.

One of the novel features in this paper is a general setup of dual graph polynomials (see Section \ref{sec:poly}) and the use of multiplied incidence matrix permanents. Special cases of these ideas have been used previously in related problems (\cite{Crump17, CrDVYe16, DV00, ZhuDelta1, Seamone, Zhu15}), but here we provide a unified general treatment for the first time.

Our paper is organized as follows. In Section \ref{sec:suff}, we prove a sufficient condition (Theorem \ref{thm:h}) for the existence of an $F$-avoiding orientation in a graph $G$, 
which is the main tool that we use throughout the paper. 
To prove our sufficient condition, we apply the Combinatorial Nullstellensatz using a polynomial similar to the one used in~\cite{akbari2020orientations}. 
In Section \ref{sec:poly},
we show that our polynomial in Section \ref{sec:suff} has a dual relationship with the traditional graph polynomial via the multiplied incidence matrix of a graph, and we show that this relationship is a special case of a more general duality relation.
In Section \ref{sec:int_rounding}, we establish a lemma about fractionally weighted subgraphs, 
and we use this lemma to
prove Theorem \ref{thm:two_thirds}.
In Section \ref{sec:reg}, we make use of a randomized approach involving the Lov\'asz Local Lemma to prove Theorem \ref{thm:reg_intro}.
Finally, in Section \ref{sec:conc}, we pose some open questions.

\section{A sufficient condition for an $F$-avoiding orientation}
\label{sec:suff}

In this section, we give a sufficient condition based on the Combinatorial Nullstellensatz (Theorem \ref{thm:CNSS}) for the existence of an $F$-avoiding orientation in a graph.

Henceforth, we make a technical change to the definition of an $F$-avoiding orientation. Given a graph $G$ with an orientation $D$ and a function $F:V(G) \rightarrow 2^{\mathbb Z}$,
rather than considering the out-degree of each vertex, we consider the \emph{imbalance} of each vertex $v \in V(G)$, 
which is defined by Mubayi, Will, and West \cite{imbalance} 
as the difference $\deg_D^+(v) - \deg_D^-(v)$. We say that $D$ is
\emph{$F$-avoiding} if $\deg_D^+(v) - \deg_D^-(v) \not \in F(v)$ for each $v \in V(G)$. 
This change is for technical reasons which will become clear later in this section.
Since the imbalance of a vertex in a given graph can be uniquely determined from its out-degree and vice versa,
and since our main result of this section is only concerned with the size of each forbidden set $F(v)$, 
our result still holds when the original definition of an $F$-avoiding orientation is used.


Before stating our condition, we need to establish some notation.
We consider a graph $G$.
We order its vertices as $v_1, \dots, v_n$, and we define an \emph{incidence matrix} $M = (m_{ve}: v \in V(G), e \in E(G))$ for $G$ with respect to the acyclic orientation on $G$ in which each edge $v_i v_j$ is oriented from $v_i$ to $v_j$ if $i < j$.
For each vertex $v_i \in V(G)$, we let $E_G^R(v_i)$ denote the edges $v_i v_j \in E(G)$ with $j > i$, and we let $E_G^L(v_i)$ denote the edges $v_i v_j \in E(G)$ with $j < i$. 
Similarly, we write $\deg_G^L(v_i) = |E_G^L(v_i)|$ and $\deg_G^R(v_i) = |E_G^R(v_i)|$.
For each edge $e \in E(G)$, we consider a variable $y_e$. 
Given an orientation $D$ of $G$, and given an edge $e = v_i v_j$ with $i < j$, we set $y_e=1$ if $e$ is oriented from $v_i$ to $v_j$ in $D$, and we set
$y_e=-1$ otherwise. 
Given $D$, we observe that 
the imbalance of each vertex $v$ can be expressed as a linear polynomial:
\[\deg_D^+(v) - \deg_D^-(v) =   \sum_{e \in E^R_G(v)} y_e - \sum_{e \in E^L_G(v)} y_e = \sum_{e \in E(G)} m_{ve} y_e.\]

Now, suppose that for each vertex $v_i \in V(G)$, we have a list $F(v_i)$ of $t_i$ integers. We would like to find an
$F$-avoiding orientation of $G$, that is, an
orientation $D$ of $E(G)$ such that for each vertex $v \in V(G)$, the imbalance of $v$ satisfies $\deg_D^+(v) - \deg_D^-(v) \not \in F(v)$. 
We fix a field $K$ of characteristic $0$, and over $K$ we define the polynomial\footnote{Here, we consider the integer values $a \in F(v_i)$ as multiples of the multiplicative identity $1 \in K$. Since $K$ has characteristic $0$, every element $a \in F(v_i)$ can be expressed as a multiple of $1$ in $K$.} 
\[f_0 = \prod_{i= 1}^n \prod_{a \in F(v_i)} \left (  \sum_{e \in E^R_G(v_i)} y_e - \sum_{e \in E^L_G(v_i)} y_e - a  \right ).\]
By the Combinatorial Nullstellensatz (Theorem \ref{thm:CNSS}), if $f_0$ has a nonzero term of degree $\deg f_0$ for which the maximum exponent of each variable $y_e$ is $1$, then
there exists a vector in $\{-1,1\}^{E(G)}$ at which $f_0$ has a nonzero evaluation, and hence
$G$ also has an $F$-avoiding orientation.

For each vertex $v \in V(G)$, we write $t_v = |F(v)|$.
We observe that $\deg f_0 = \sum_{v \in V(G)} t_v$ and $t_i = t_{v_i}$ for each $v_i \in V(G)$.
Since each term of degree $\deg f_0 = \sum_{v \in V(G)} t_v$ in the monomial expansion of $f_0$ is obtained by choosing exactly one $y_e$ term from each factor
$\left (  \sum_{e \in E^R_G(v_i)} y_e - \sum_{e \in E^L_G(v_i)} y_e - a  \right )$, 
the terms of maximum degree in the monomial expansion of $f_0$ are exactly the same as the terms of maximum degree in the monomial expansion of 
\[   f :=  \prod_{i= 1}^n  \left (  \sum_{e \in E^R_G(v_i)} y_e - \sum_{e \in E^L_G(v_i)} y_e \right )^{t_i}
=  \prod_{v \in V(G)}   \left (   \sum_{e \in E(G)} m_{ve} y_e \right )^{t_v}
.\]
Therefore, rather than working with $f_0$, we work with the simpler polynomial $f$.
By our previous discussion, if $f$ has a nonzero term of degree 
$\sum_{v \in V(G)} t_v$ for which the maximum exponent of each variable $y_e$ is $1$, then
$G$ has an $F$-avoiding orientation.



Given an edge set $A \subseteq E(G)$, we write $y^A = y_{e_1} y_{e_2} \dots y_{e_{|A|}}$, where $A = \{e_1, e_2, \dots, e_{|A|}\}$.
We say that a monomial $P \in K[y_e:e \in E(G)]$ is \emph{square-free} if for each edge $e \in E(G)$, $y_e^2$ does not divide $P$. Observe that after removing its coefficient, a square-free monomial $P \in K[y_e: e \in E(G)]$ is of the form $y^A$ for some edge set $A \subseteq E(G)$. Therefore, if the monomial expansion of $f$ contains a square-free term with a nonzero coefficient, then by the Combinatorial Nullstellensatz, $G$ has an $F$-avoiding orientation.

Now, we are ready to state our sufficient condition for when $G$ has an $F$-avoiding orientation. Although we prove the following theorem within the framework of forbidden vertex imbalances, it is easy to see that the theorem also holds in the setting of forbidden vertex out-degrees.

\begin{theorem}
\label{thm:h}
Let $F:V(G) \rightarrow 2^{\mathbb Z}$ be an assignment of forbidden imbalances for a graph $G$.
Suppose that there exists an ordering of $V(G)$ and a spanning subgraph $H$ of $G$ such that for each vertex $v \in V(G)$, \[ |F(v)| \leq \deg^L_G(v) - 2 \deg^L_H(v) + \deg^R_H(v).\]
Then $G$ has an $F$-avoiding orientation.
\end{theorem}

\begin{proof}
In the proof we use the notation introduced above. We also set $E = E(G)$.
In order to prove that $G$ has an $F$-avoiding orientation, we show that the monomial expansion of $f$ over $K$ contains a square-free term and then apply the Combinatorial Nullstellensatz as described above.
Given an edge set $A \subseteq E$, we say that $y^A$ is in the \emph{support} of $f$ if the monomial $y^A$ has nonzero coefficient in the expansion of $f$. If $y^A$ is in the support of $f$, we also say that $A$ is in the support of $f$, and we write $A \in \supp(f)$.
For each $j$ ($1 \leq j \leq n$), we write 
\[   f_j =  \prod_{i= 1}^j  \left (  \sum_{e \in E^R_G(v_i)} y_e - \sum_{e \in E^L_G(v_i)} y_e \right )^{t_i},\]
and we observe that $f = f_n$. 

We prove the following stronger claim:
\medskip

For each $j$ ($1 \leq j \leq n$), there exists an edge set $A_j \subseteq E$ such that
\begin{enumerate}
\item[(a)] \label{i:a} ${A_j} \in \supp(f_j)$,
\item[(b)] \label{i:c}  if $k > j$, then $A_j \cap E^L_G(v_k) \subseteq E_H^L(v_k)$. 
\end{enumerate}
\medskip

By setting $j = n$ in the claim, we obtain an edge set $A_n \subseteq E$ which is in the support of $f_n = f$. 
Since $A_n$ is an edge set (and not a multiset), $y^{A_n}$ is square-free and thus satisfies the theorem.
When proving the claim, we work in the quotient ring \[  K' = K[y_e : e \in E]/ \langle y_e^2 : e \in E  \rangle,\] 
where $\langle 
y_e^2 : e \in E
\rangle$ is the ideal generated by the squares of the variables $y_{e}$. This allows us to ignore all terms in the expansion of each $f_j$ divisible by a square, which is desirable, as we are only interested in finding a square-free term in the expansion of $f$.

We prove the claim by induction on $j$. When $j = 1$, we observe that 
\[
    f_1 =  \left ( \sum_{e \in E_G^R(v_1)} y_e \right )^{t_1}  = t_1! \sum_{\substack {A \subseteq E^R_G(v_1) \\ |A| = t_1}}  y^A .
\]
Since $t_1 = |F(v_1)| \leq  \deg^L_G(v_1) - 2 \deg^L_H(v_1) + \deg^R_H(v_1) = \deg^R_H(v_1)$,
we can choose $A_1$ to be any set of $t_1$ edges of $H$ incident with $v_1$. Hence the claim holds for $j=1$.

Now, let $2 \leq j \leq n$, and suppose the claim holds for $j-1$.
That is, suppose that we have a set $A_{j-1}$ satisfying $y^{A_{j-1}} \in \supp(f_{j-1})$ 
and in particular
\[A_{j-1} \cap E_G^L(v_{j}) \subseteq E_H^L(v_{j}).\]
We would like to show that the claim holds for $j$. First, we will show that $f_j$ has a nonzero expansion in our quotient ring, implying the existence of a square-free monomial in the expansion of $f_j$ over $K'$. Then, we will show that the square-free monomial that we have found in the support of $f_j$ satisfies condition (b). As
\[f_j = f_{j-1} \cdot \left ( \sum_{e \in E^R_G(v_j)} y_e - \sum_{e \in E^L_G(v_j)}y_e \right )^{t_j}, \]
we can express $f_j$ in the following form: 
\begin{equation}
\label{eqn:fj}
 f_j = f_{j-1} \cdot \sum_{a = 0}^{t_j} {\binom{t_j}{a}} (-1)^{t_j - a}  \left ( \sum_{e \in E^L_G(v_j)} y_e \right )^{t_j - a} 
 \left(\sum_{e \in E^R_G(v_j)} y_e \right )^a.
\end{equation}
When the sum in (\ref{eqn:fj}) is restricted to a single value $a$, each monomial in the expansion of (\ref{eqn:fj}) has exactly $a$ variables $y_e$ for which $e \in E_G^R(v_j)$. 
Since none of these variables $y_e$ appears in $f_{j-1}$, any nonzero term in the expansion of (\ref{eqn:fj}) for a fixed value of $a$ is also a nonzero term in the expansion of (\ref{eqn:fj}) when the sum is taken over all values of $a$.

We restrict our attention to terms in the expansion of (\ref{eqn:fj}) that occur when $a = m:= \min \{t_j, \deg_H^R(v_j)\}$ in the sum. Hence, we only consider the expansion of
\begin{equation}
\label{eqn:afixed}
f_{j-1} \cdot \left( \sum_{e \in E^L_G(v_j)} y_e \right )^{t_j - m}   
\left(\sum_{e \in E^R_G(v_j)} y_e \right )^{m}.
\end{equation}

We fix a set $A' \subseteq E_H^R(v_j)$ with $m$ edges and observe that ${A'} \in  \supp \left ( \sum_{E_G^R(v_j)} y_e \right ) ^m$.
Furthermore, we write $B = A_{j-1} \setminus E_G^L(v_j)$, and we write
\[f_{j-1} = y^B g + r,\]
where $g$ and $r$ are polynomials, and $r$ does not have any terms divisible by $y^B$. Observe that every term of $g$ is a constant times a monomial $P$ of degree $d: = \deg (f_{j-1}) - |B| = |A_{j-1} \setminus B| = |A_{j-1} \cap E_G^L(v_j) | $ corresponding to a subset of $d$ edges of $E_G^L(v_j)$. 
Observe also that $A_{j-1} \cap E^L_G(v_j)$ is in the support of $g$.

Now, we would like to show that
\begin{equation}
\label{eqn:gneq0}
  g \cdot \left (  \sum_{e \in E^L_G(v_j)} y_e \right )^{t_j - m}  \neq 0 ,
\end{equation}
in the quotient ring $K'$. 
Observe that if 
(\ref{eqn:gneq0}) holds in $K'$,
 then since $B \cap E^L_G(v_j) = \emptyset$, it also holds that
\begin{equation}
\label{eqn:fj1}
  f_{j-1} \cdot \left (  \sum_{e \in E^L_G(v_j)} y_e \right )^{t_j - m}  = (y^B g + r)\left (  \sum_{e \in E^L_G(v_j)} y_e \right )^{t_j - m}
\end{equation}
is nonzero in $K'$. Also, since no $y_e$ with $e \in E^R_G(v_j)$ divides (\ref{eqn:fj1}) in $K'$, it follows that (\ref{eqn:afixed}) has a nonzero monomial expansion in $K'$ (and hence a square-free term) whenever (\ref{eqn:gneq0}) holds. Therefore, in order to prove condition (a) of our claim, it is enough to prove (\ref{eqn:gneq0}).


If $m = t_j$, then (\ref{eqn:gneq0}) clearly holds, and $  g \cdot \left (  \sum_{e \in E^L_G(v_j)} y_e \right )^{t_j - m} $ has a square-free monomial $y^S$
corresponding to an edge subset $S \subseteq E_G^L(v_j)$. Then, the edge set $A_j = A' \cup B \cup S$ corresponds to a square-free monomial in the support of $f_j$. By construction, the only edges of  $A_j \setminus A_{j-1}$ 
that do not belong to $E(H)$ are those that belong to $E^L_G(v_j)$, which implies condition (b).
Hence, the induction step is complete, and the claim is proven.

Otherwise, we have
$m = \deg^R_H(v_j)$.
We write $E^L = E^L_G(v_j)$ and $b = d + t_j - \deg_H^R(v_j)$. 
We also expand $g$ as 
\[ g = \sum_{\substack{D \subseteq E^L \\ |D| = d}} c_D y^D.\]
Then,
\begin{eqnarray}
\nonumber
 g\left(\sum_{e \in E^L} y_e \right)^{t_j - \deg_H^R(v_j)} &=& (t_j - \deg_H^R(v_j))! \left ( \sum_{\substack{D \subseteq E^L \\ |D| = d}}    c_D y^D \right ) \left (\sum_{\substack{U \subseteq E^L \\ |U| = t_j - \deg_H^R(v_j)}} y^U \right ) \\
 &=& (t_j - \deg_H^R(v_j))! \sum_{\substack{Y \subseteq E^L  \\ |Y| = b}}  \left (     \sum_{\substack{   D \subseteq Y \\ |D| = d }} c_D \right )  y^Y. \label{eq:g expands}
\end{eqnarray}
Hence, in order to show that (\ref{eqn:gneq0}) holds and prove the first condition of our induction hypothesis, it is enough to show that the expression above is nonzero in our quotient ring.

The expression (\ref{eq:g expands}) may be expressed as follows. Let $Q$ be the matrix whose rows are indexed by the $b$-subsets of $E^L$ and whose columns are indexed by the $d$-subsets of $E^L$. Given a set $Y \subseteq E^L$ of size $b$ and a set $D \subseteq E^L$ of size $d$, we let the $(Y,D)$-entry of $Q$ be $1$ if $D \subseteq Y$ and equal $0$ otherwise. In other words, $Q$ is a set-inclusion matrix.
Then, the coefficients in (\ref{eq:g expands}) are equal to the entries in the vector
\begin{equation}
\label{eqn:matrix}
(t_j - \deg_H^R(v_j))! \, Q \tau, 
\end{equation}
where $\tau$ is the column vector with entries indexed by subsets $D \subseteq E^L$ of size $d$ and whose $D$-entry equals $c_D$. Since $g$ is nonzero in our quotient ring, we have that $\tau \neq 0$.
Thus, to prove that (\ref{eqn:gneq0}) holds, it suffices to show that (\ref{eqn:matrix}) is nonzero.

Now, since $d = |A_{j-1} \cap E_G^L(v_j) |$, it follows from condition (b) 
of the induction hypothesis that $d \leq \deg^L_H(v_j)$. Therefore, $b \leq t_j + \deg^L_H(v_j) - \deg^R_H(v_j)$, which by our initial assumption is at most $|E^L| - \deg^L_H(v_j) \leq |E^L| - d$. Therefore, since $d \leq b \leq |E^L| - d$, it follows that $\binom{|E^L|}{b} \geq \binom{|E^L|}{d}$, and hence it follows from a classical result of Gottlieb \cite[Corollary 1 (b)]{gottlieb1966certain} about set-inclusion matrices that $Q$ has full column rank. Hence, the matrix product in (\ref{eqn:matrix}) is nonzero, and thus (\ref{eqn:gneq0}) holds. 

Therefore,
 $  g \cdot \left (  \sum_{e \in E^L_G(v_j)} y_e \right )^{t_j - m} $ again has a square-free monomial $y^S$
corresponding to an edge subset $S \subseteq E_G^L(v_j)$.
We let $A_j = A' \cup B \cup S$, and then $A_j$ corresponds to a square-free monomial in the support of $f_j$.

Finally, we check that our edge set $A_j$
satisfies condition (b).
By construction, the only edges of $A_j \setminus A_{j-1}$ 
that do not belong to $E(H)$ are those that belong to $E^L_G(v_j)$, which implies (b). Hence, the induction step is complete, and the claim is proven. As stated at the beginning of the proof, the claim together with the Combinatorial Nullstellensatz implies our theorem.
\end{proof}

\section{Dual polynomials}
\label{sec:poly}

In this section, we take a brief aside to show that our polynomial $f$ used in the proof of Theorem \ref{thm:h} is related to the traditional graph polynomial and that this relationship is a special case of a more general duality relation. As a corollary, we obtain a simplified proof of Theorem \ref{thm:2}.

We consider the following setup. 
Let $A = (a_{ij} : 1 \leq i \leq n, 1 \leq j \leq m)$ be an $n \times m$ matrix, and let $\alpha = (\alpha_1, \dots, \alpha_n)$ and $\beta = (\beta_1, \dots, \beta_m)$ be nonnegative integer vectors.
Let $A^{\alpha, \beta}$ be obtained from $A$ by making exactly $\alpha_i$ copies of each row $i$ in $A$ and then subsequently making exactly $\beta_j$ copies of each column $j$. 
We fix a field $K$ and the ring $K[x_i, y_j: 1 \leq i \leq n, 1 \leq j \leq m]$.
Given vectors $\alpha$ and $\beta$ as above, we define the \emph{dual polynomials}
\begin{eqnarray*}
g &=& \prod_{i = 1}^n \left ( \sum_{j = 1}^m a_{ij} y_j \right )^{\alpha_i}, \\
g^* &=& \prod_{j= 1}^m \left ( \sum_{i = 1}^n a_{ij} x_i \right )^{\beta_j}.
\end{eqnarray*}
Note that when only one polynomial is given, its dual is not uniquely determined, since the correspondence between $g$ and $g^*$ depends on the vectors $\alpha$ and $\beta$.
We say that $g$ and $g^*$ are duals of each other, as the factors of $g$ are obtained from the rows of $A$, while the factors of $g^*$ are obtained from the columns of $A$. This relationship is reminiscent to the relationship between dual linear programs, whose respective constraints are given by the rows and the columns of a common matrix.

We write $x^{\alpha} = \prod_{i = 1}^n x_i^{\alpha_i}$ and $y^{\beta} = \prod_{j = 1}^m y_j^{\beta_j}$. We denote the coefficient of a monomial $y^{\beta}$ in the polynomial $g$ by $\coeff(y^{\beta},g)$. For a vector $u = (u_1, \dots, u_k)$, we write $\|u\|_1 = \sum_{i = 1}^k |u_i|$. Finally, if $M$ is a square matrix, then we denote by $\perm(M)$ its permanent. We have the following result.

\begin{theorem}
\label{thm:duals}
If\/ $\| \alpha \|_1 = \| \beta \|_1$, then 
$( \prod_{j = 1}^m \beta_j ! ) \coeff \left(y^{\beta} ,g \right) =  \left ( \prod_{i = 1}^n \alpha_i ! \right ) \coeff(x^{\alpha},g^*) = \perm(A^{\alpha,\beta}) $. 
\end{theorem}

\begin{proof}
We first show that $ ( \prod_{j = 1}^m \beta_j !  ) \coeff \left (y^{\beta} ,g \right ) = \perm(A^{\alpha,\beta}) $.
Consider the monomial expansion of $g$ containing $m^{\|\alpha \|_1}$ terms, each of which is obtained by choosing a term from each of the $\| \alpha \|_1$ 
factors $\sum_{j = 1}^m a_{ij} y_j$ of $g$ and taking the product of these terms.
Each of the factors $\sum_{j = 1}^m a_{ij} y_j$ of $g$ can be associated with a distinct row in the matrix $A^{\alpha,\mathbf 1}$ (where $\mathbf 1$ refers to the all-1 vector of length $m$). Therefore, each term $P$ in the monomial expansion of $g$ divisible by $y^{\beta}$ corresponds to a choice of $\|\alpha\|_1$ entries from $A^{\alpha,\mathbf 1}$ with one entry from each row and $\beta_j$ entries in each column $j$,
and
the coefficient of $P$ is equal to the product of all $\|  \alpha   \|_1$ of these entries. 
Hence, $\coeff(y^{\beta}, g)$ is equal to $\sum_{\sigma} a_{\sigma}$, where $\sigma$ runs over all choices of $\|  \alpha   \|_1$ entries from $A^{\alpha,\mathbf 1}$ consisting of one entry from each row and $\beta_j$ entries from each column $j$, and $a_{\sigma}$ is the product of all of the values from $A^{\alpha,\mathbf 1}$ contained in $\sigma$. 
Furthermore, a choice of $\beta_j$ entries from column $j$ of $A^{\alpha,\mathbf 1}$ can be represented in $\beta_j!$
ways as a choice of one entry from each of the $\beta_j $ columns of $A^{\alpha,\beta}$ corresponding to column $j$ of $A^{\alpha,\mathbf 1}$. Therefore, 
$( \prod_{j = 1}^m \beta_j!  )  \coeff(y^{\beta}, g) = \sum_{\tau} a_{\tau}$, where $\tau$ runs over all choices of $\|\alpha\|_1$ entries from $A^{\alpha,\beta}$ consisting of one entry from each row and one entry from each column, and $a_{\tau}$ is the product of all of the values from $A^{\alpha,\beta}$ contained in $\tau$. 
Equivalently, $\left ( \prod_{i=1}^m \beta_j! \right ) \coeff(y^{\beta}, g) = \perm(A^{\alpha,\beta})$.

By using the transpose of $A$ and swapping $\alpha$ and $\beta$, the same argument shows that 
$\left( \prod_{i = 1}^n \alpha_i ! \right) \coeff( x^{\alpha},g^* ) = \perm((A^T)^{\beta,\alpha}) = \perm(A^{\alpha,\beta}) $.
\end{proof}

We note that considering the permanent of a matrix appears naturally when using the polynomial method. The first such appearances may have been used by DeVos et al.\ \cite{DV00, CrDVYe16, Crump17}. It was also used in relation to the problem of total weight choosability (see e.g.~\cite{ZhuDelta1} and \cite{Seamone}, \cite{Zhu15}). In fact, in \cite{Zhu15}, Zhu similarly considers the permanent of a matrix obtained by making copies of the rows or columns of a smaller matrix.

In one special case of this matrix setting, $A$ is the incidence matrix $M$ of a graph $G$, and $\beta \in \{0,1\}^{|E(G)|}$ indexes an edge set $E' \subseteq E(G)$. If $f$ is defined as in Section \ref{sec:suff} and $\coeff(y^{\beta}, f) \neq 0$, then the Combinatorial Nullstellensatz tells us that $G$ has an $F$-avoiding orientation when $|F(v)| \le t_v$ for each vertex $v$ of $G$.
Furthermore, the polynomial dual to $f$ is $f^* = \prod_{uv \in E'} (x_u - x_v)$, 
which is the traditional graph polynomial of $G[E']$.
A famous result of Alon and Tarsi \cite{AT} tells us that if $D$ is an orientation of a graph $H$ satisfying $\deg_D^+(v) = t_v$ at each vertex $v \in V(H)$,
then
\begin{equation}
\tag{AT} \label{eqn:AT}
     \biggl| \coeff\biggl(\,\prod_{v \in V(G)} x_v^{t_v} , f^* \biggr) \biggr| = |EE(D) - EO(D)|.
\end{equation}
Therefore, if there exists an Alon-Tarsi orientation $D$ of $G[E']$ in which $\deg_D^+(v) = t_v$ for each $v \in V(G)$, then (\ref{eqn:AT}) and Theorem \ref{thm:duals} tell us that 
\[ \biggl( \prod_{v \in V(G)} t_v! \biggr)^{-1} \biggl| \coeff \biggl(y^{\beta}, f \biggr) \biggr| = \biggl| \coeff \biggl(\,\prod_{v \in V(G)} x_v^{t_v} , f^* \biggr) \biggr| =   |EE(D) - EO(D)| \neq 0.\]
Hence, the existence of the Alon-Tarsi orientation $D$ on $G[E']$ implies that $G$ has an $F$-avoiding orientation whenever $|F(v)| \le t_v$ for each $v \in V(G)$, proving Theorem \ref{thm:2}.
Theorem \ref{thm:duals} gives us a sufficient condition for the existence of a nowhere-zero $p$-flow, for each prime $p$. In fact, the theorem gives us a stronger conclusion, which we describe below.


Recall that a function $\phi: E(G) \rightarrow \{1, \dots, p-1\}$ on a graph $G$ yields a nowhere-zero $p$-flow if and only if the congruence
$\sum_{e \in E^+(v)} \phi(e) - \sum_{e \in E^-(v)}  \phi(e) \equiv 0 \pmod p$
holds at each vertex $v \in V(G)$. 
Generalizing the notion of a nowhere-zero flow, we say that a graph $G$ is \emph{$\mathbb{Z}_p$-connected} if for every function $b:V(G)\to \mathbb{Z}_p$ with $\sum_{v\in V(G)} b(v)=0$, there exists $\phi: E(G) \to \mathbb{Z}_p\setminus \{0\}$ such that 
\begin{equation}
\label{eq:b-flow}
   \sum_{e \in E^+(v)} \phi(e) - \sum_{e \in E^-(v)} \phi(e) = b(v)
   \end{equation}
for each $v \in V(G)$.
Since $\sum_{v \in V(G)} \left ( \sum_{e \in E^+(v)} \phi(v) - \sum_{e \in E^-(v)} \phi(v) \right ) \equiv 0 \pmod p$,
it follows that if (\ref{eq:b-flow}) is satisfied at every vertex $v \in V(G) \setminus \{u\}$ for some $u \in V(G)$, then
(\ref{eq:b-flow}) also holds at $u$.


\begin{theorem}
Let $p\ge3$ be a prime, and let $G$ be a graph. Let $G^{(p-2)}$ be obtained from $G$ by replacing each edge with $p-2$ parallel edges, and let $H$ be a spanning subgraph of $G^{(p-2)}$ with exactly $(p-1)(|V(G)|-1)$ edges. If, for some $u \in V(H)$, $H$ has an orientation $D$ in which $\deg_D^+(v) = p-1$ for each $v \in V(G) \setminus \{u\}$
and such that $EE(D) - EO(D) \not \equiv 0 \pmod p$, then $G$ is $\mathbb{Z}_p$-connected.
\end{theorem}

\begin{proof}
Consider an arbitrary orientation of the edges of $G$.
From this orientation, define the incidence matrix $M = (m_{ve}: v \in V(G), e \in E(G))$ of $G$ and consider any $b:V(G)\to \mathbb{Z}_p$ with $\sum_{v\in V(G)} b(v)=0$.  We fix a vertex $u\in V(G)$, the field $\mathbb Z_p$, and a ring $\mathbb Z_p[x_v, y_e: v \in V(G), e \in E(G)]$ in which we define the polynomial
\[
    g_0 = \prod_{v \in V(G) \setminus \{u\}} \ \prod_{a \in \mathbb{Z}_p\setminus\{b(v)\}} \left ( \sum_{e \in E(G)} m_{ve} y_e - a \right ) .
\]
It is straightforward to show that there is a function $\phi: E(G) \to \mathbb{Z}_p\setminus \{0\}$ satisfying (\ref{eq:b-flow}) if and only if there is an assignment from the set $\mathbb{Z}_p\setminus \{0\}$ to each variable $y_e$ that gives $g_0$ a nonzero evaluation.
Furthermore, every term of maximum degree in the expansion of $g_0$ is also a maximum-degree term in the expansion of 
\[
   g = \prod_{v \in V(G) \setminus \{u\}} \left(\sum_{e \in E(G)} m_{ve} y_e \right)^{p-1} = \prod_{v \in V(G)} \left( \sum_{e \in E(G)} m_{ve} y_e  \right)^{\alpha_v}, 
\]
where $\alpha_v = p-1$ for each $v \in V(G) \setminus \{u\}$, and $\alpha_u = 0$.
We let $\beta \in \{0,\dots,p-2\}^{|E(G)|}$ be a vector indexing which edges of $G$, and how many copies of each edge, belong to $H$.
Since the maximum degree terms in $g$ and $g_0$ are equal, it follows that if $\coeff(y^{\beta}, g) \neq 0$, then
by the Combinatorial Nullstellensatz, $g_0$ has an assignment for each $y_e$ from $\mathbb{Z}_p\setminus \{0\}$ 
that makes $g_0$ nonzero.

Note that $\sum_{v\in V(G)} \alpha_v = (p-1)(|V(G)|-1) = |E(H)| = \sum_{e\in E(G)} \beta_e$. Thus we can apply
Theorem \ref{thm:duals}, which shows that $\coeff(y^{\beta}, g) \neq 0$ if and only if $\coeff(x^{\alpha},g^*) \neq 0$, where
\[ g^* = \prod_{e \in E(G)} \left ( \sum_{v \in V(G)} m_{ve} x_v  \right )^{\beta_e} = \prod_{e \in E(H)} \left ( \sum_{v \in V(G)} m_{ve} x_v  \right ). \]
However, since $g^*$ is the graph polynomial of $H$, and since $D$ is an orientation of $H$ satisfying $\deg_D^+(v) = \alpha_v$ at each $v \in V(H)$ and $EE(D) - EO(D) \neq 0$ in $\mathbb Z_p$, it follows from (\ref{eqn:AT}) that $|\coeff(x^{\alpha},g^*)| =  | EE(D) - EO(D)| \neq 0$. Therefore, an assignment from $\mathbb{Z}_p\setminus\{0\}$ to each $y_e$  giving $g_0$ a nonzero evaluation exists.  This implies that $G$ is $\mathbb{Z}_p$-connected.  
\end{proof}

\section{A step toward Conjectures \ref{conj:1} and \ref{conj:2}}
\label{sec:int_rounding}

In this section, we prove Theorem \ref{thm:two_thirds}, which establishes relaxations of Conjectures \ref{conj:1} and \ref{conj:2} with slightly smaller constant factors. 
Our main tool is Theorem \ref{thm:h} along with a fractional rounding method.

Many graph theory problems can be described as integer programming problems. Often, we find a fractional solution, and then use rounding to obtain an integer solution that approximately solves the original problem. The following lemma shows how this is done.

\begin{lemma}\label{lem:Integer_Rounding}
Given a graph $G = (V,E)$, let $M=(m_{ve}:v\in V, e\in E)$ be a real-valued matrix in which $m_{ve} \neq 0$ only if $v\in e$. 
Let $\mathbf{y}\in [0, 1]^E$ be a vector, and let $\mathbf{x}=M\mathbf{y}$. Then, there exists a $01$-vector $\mathbf{y}'\in \{0, 1\}^E$ such that  $\mathbf{x}'=M\mathbf{y}'$ satisfies $\mathbf{x}'_v \geq  \mathbf{x}_v-b_v$ 
for each $v\in V(G)$,  where $b_v=\max\{|m_{ve}|:e\in E\}$. 
Furthermore, we may choose $\mathbf{y}'$ so that $\mathbf{x}'_v > \mathbf{x}_v-b_v$ whenever $b_v > 0$.
\end{lemma}

\begin{proof}
Given $\mathbf{t}\in [0,1]^E$, we write
$fr(\mathbf{t})=\{e: 0<\mathbf{t}_e<1\}$. Let $\mathbf{z}\in [0,1]^E$ be a solution of $M\mathbf{z}\ge \mathbf{x}$ with minimum size of $fr(\mathbf{z})$. We know that $\mathbf{z}$ exists, because $\mathbf z = \mathbf y$ satisfies the inequality. 

First we show that the subgraph $H$ of $G$ induced by the edges in $fr(\mathbf{z})$ is acyclic. If $H$ contains a cycle $C$, then we claim that there exists a non-trivial vector $\mathbf{\alpha}\in \mathbb{R}^E$ with its support contained in $E(C)$  such that $M\mathbf{\alpha}\ge 0$.
Indeed, let $M'$ be obtained from $M$ by taking the columns corresponding to $E(C)$ and rows corresponding to $V(C)$. 
As $C$ is a cycle, $M'$ is a square matrix. Since $m_{ve} = 0$ whenever $v \not \in e$, finding our vector $\mathbf{\alpha}$
with support contained in $E(C)$ is equivalent to finding a nontrivial vector $\mathbf{ \alpha }'$ satisfying $M'\mathbf{\alpha}' \geq 0$. 
If $M'$ is singular, then $\mathbf{\alpha}'$ may be taken from the nullspace of $M'$. On the other hand, if $M'$ is invertible, then $\mathbf{\alpha}'$ may be taken as the unique solution to $M' \mathbf{\alpha}' = e_u$, where $e_u$ is the unit vector corresponding to some vertex $u \in V(C)$. In both cases, we find our vector $\mathbf{\alpha}'$ and hence also our vector $\mathbf{\alpha}$.  

Now, we can choose a value $p>0$ such that $\mathbf{z}+p\mathbf{\alpha}\in [0,1]^E$ and
$|fr(\mathbf{z}+p\mathbf{\alpha})|<|fr(\mathbf{z})|$. 
Then, we have $M(\mathbf{z}+p\mathbf{\alpha})=M\mathbf{z}+pM\mathbf{\alpha}\ge M\mathbf{z}$, which contradicts our choice of $\mathbf{z}$.
Thus, we conclude that the subgraph $H$ of $G$ induced by $fr(\mathbf z)$ is acyclic.

Now, if $H$ is non-empty, then $H$ has a leaf $v$ incident to a single edge $e=vu \in E(H)$. If we change the value of $\mathbf{z}_e$, only the coordinates $v$ and $u$ of $M\mathbf z$ change. We can change $\mathbf z_e$ to $0$ or $1$ so that the $u$-coordinate of $M \mathbf z$ increases, while the $v$-coordinate decreases by less than $b_v$. 
We consider $H-e$ and repeat this process until $fr(\mathbf z)$ is empty,
and we write $\mathbf{y}'$ for the new integer vector we obtain from $\mathbf z$.
We write $\mathbf{x}'=M\mathbf{y}'$, and by our observation above, $\mathbf{x}'_v \geq  \mathbf{x}_v-b_v$ for each $v\in V(G)$, with a strict inequality whenever $b_v > 0$.
\end{proof}

As a warm-up for our proof of Theorem \ref{thm:two_thirds}, we first prove an easier result, which is an approximation of Conjecture \ref{conj:1}. By choosing an orientation of $G$ that makes the in-degree and out-degree as equal as possible at each vertex,
Theorem \ref{thm:two_thirds} implies the following result with only a constant error.

\begin{theorem}
\label{thm:third}
Let $G$ be a graph, and let $F:V(G) \rightarrow 2^{\mathbb N}$. If, for each $v\in V(G)$ we have,
\[ |F(v)|\le \frac{1}{3}\deg(v)-1,\]
then $G$ has an $F$-avoiding orientation.
\end{theorem}
\begin{proof}

By Theorem~\ref{thm:h}, we just need to show that there exist an ordering of vertices $v_1,\dots, v_n$, and a subgraph $H$ of $G$ such that for each vertex $v\in V(G)$, it holds that 
\[\deg^L_G(v) - 2 \deg^L_H(v) +\deg^R_H(v)\ge \lfloor \tfrac{1}{3}\deg(v)\rfloor-1.\]
We fix an ordering $v_1, \dots, v_n$ of the vertices of $G$. Also, we define a matrix $M=(m_{ve}: v\in V(G), e\in E(G))$ as follows. For each edge $e= v_iv_j$ with $i<j$, we let the column of $M$ corresponding to $e$ have entries $1$ and $-2$ in the rows corresponding to $v_i$ and $v_j$, respectively, and we let all other entries in the column be $0$.
Our task is equivalent to finding a
$01$-vector $\mathbf{y}'\in \{0,1\}^E$ such that the vector $\mathbf{x}' = M\mathbf{y}'$ satisfies $\mathbf{x}'_v \geq \lfloor\frac{1}{3}\deg(v)\rfloor-\deg^L_G(v)-1$. We begin with a vector $\mathbf y \in \mathbb R^E$ with $1/3$ in each of its entries, and we observe that $\mathbf x = M \mathbf y$ satisfies 
\[
 \mathbf{x}_v = - \frac{2}{3} \deg^L_G(v) + \frac{1}{3}\deg^R_G(v) = \frac{1}{3}\deg_G(v) - \deg_G^L(v)
\]
for each $v \in V(G)$. Hence, by Lemma \ref{lem:Integer_Rounding}, there exists $\mathbf{y}'\in \{0,1\}^E$ such that the vector $\mathbf{x}' = M\mathbf{y}'$ satisfies 
\[\mathbf{x}'_v > \frac{1}{3}\deg(v)-\deg^L_G(v)-2
\]
at each vertex $v \in V(G)$. Equivalently, $\mathbf{x}'_v \geq \lfloor\frac{1}{3}\deg(v)\rfloor-\deg^L_G(v)-1$ at each $v$, completing the proof.
\end{proof}

Using a similar method, we also prove Theorem \ref{thm:two_thirds}, which gives a  $\frac{2}{3}$-approximation of Conjecture~\ref{conj:2}.


\begin{proof}[Proof of Theorem \ref{thm:two_thirds}]
Again, by Theorem~\ref{thm:h}, we just need to show that there exists an ordering $v_1,\dots, v_n$ of $V(G)$, as well as a subgraph $H$ of $G$ such that for each vertex $v\in V(G)$, it holds that 
\[\deg^L_G(v) - 2 \deg^L_H(v) +\deg^R_H(v)\ge \lfloor \tfrac{2}{3}\deg^+_D(v)\rfloor-1.\]
We fix an ordering $v_1, \dots, v_n$ of $V(G)$, and then define a matrix $M$ as in the proof of Theorem \ref{thm:third}.
To find a subgraph $H$ satisfying the above inequality 
for each $v\in V(G)$, it is equivalent to find an integer vector $\mathbf{y}'\in \{0,1\}^E$ such that the vector $\mathbf{x}' = M\mathbf{y}'$ satisfies
$\mathbf{x}'_v \ge \lfloor\frac{2}{3}\deg^+_D(v)\rfloor-\deg^L_G(v)-1$ for each $v$.

We say that an edge $v_iv_j$ with $i<j$ is a \emph{forward edge} in $D$ if it is oriented from $v_i$ to $v_j$; otherwise, we call it a
\emph{backward edge} in $D$. 
We write $\deg^{L+}(v)$ for the number of backward edges outgoing from $v$ in $D$, and we write $\deg^{L-}(v)$ for the number of forward edges incoming to $v$ in $D$.
Similarly, we write $\deg^{R+}(v)$ for the number of forward edges outgoing from $v$ in $D$, and $\deg^{R-}(v)$ for the number of backward edges incoming to $v$ in $D$.

Let $v_1, \dots, v_n$ be an ordering of $V(G)$ with the minimum possible 
number of forward edges.
We argue that for each $v \in V(G)$, it holds that
$\deg^{L+}(v)\ge \deg^{L-}(v)$. 
Indeed, if this inequality does not hold for $v$, then we can move $v$ to the first position in our vertex ordering and decrease the number of forward edges in $D$ by $\deg^{L-}(v)-\deg^{L+}(v)> 0$, a contradiction.

Now, let $\mathbf{y}_e=\frac{2}{3}$ for each forward edge $e$ in $D$,
and let $\mathbf{y}_e=0$ for each backward edge $e$ in $D$.
Letting $\mathbf{x}=M\mathbf{y}$, we have
\begin{align*}
    \mathbf{x}_v+\deg^L_G(v) & -\frac{2}{3}\deg^+_D(v)\\
    &=\, \frac{2}{3}(\deg^{R+}(v)-2\deg^{L-}(v)) + (\deg^{L-}(v)+\deg^{L+}(v)) - \frac{2}{3}(\deg^{L+}(v)+\deg^{R+}(v))\\
    &=\, \frac{1}{3}(\deg^{L+}(v)-\deg^{L-}(v))\ge 0.
\end{align*}
Therefore, for each $v \in V(G)$, we have $\mathbf{x}_v \ge \frac{2}{3}\deg^+_D(v)-\deg^L_G(v)$. 
Since $|m_{ve}|\le 2$ for each entry $m_{ev}$ in $M$, it follows from 
Lemma~\ref{lem:Integer_Rounding}
that there exists an integer vector $\mathbf{y}' \in \{0,1\}^E$ such that the vector
$\mathbf{x}' = M\mathbf{y}'$ satisfies
$\mathbf{x}_v' > \frac{2}{3}\deg^+_D(v)-\deg^L_G(v)-2$
for each $v \in V(G)$.
Equivalently, for each vertex $v$, we have $\mathbf{x_v}'\ge \lfloor\frac{2}{3}\deg^+_D(v)\rfloor-\deg^L_G(v)-1$, 
completing the proof.
\end{proof}

\section{Graphs with a relaxed regularity condition}
\label{sec:reg}

In this section, we prove Theorem \ref{thm:reg_intro}, which holds for large graphs whose maximum degree is subexponential in terms of their minimum degree. 
For our proof, we need some preliminary lemmas and definitions. We use the Chernoff bound, a probabilistic tool that can be found in many textbooks (e.g.~\cite[Chapter 4]{mitzenmacher2017probability}).

\begin{lemma}
\label{lem:chernoff}
Let $X$ be a binomially distributed variable with parameters $n$ and $p$. Let $\mu = np$, and let $0 < \delta \leq 1$. Then,
$$\Pr(X > (1 + \delta) \mu) \leq \exp \left( - \frac{1 }{3 } \delta^2 \mu  \right)$$
and
$$\Pr(X< (1 - \delta) \mu) \leq \exp \left( - \frac{1}{3} \delta^2 \mu \right).$$
\end{lemma}

Mitzenmacher and Upfal \cite{mitzenmacher2017probability} point out that for the first statement of the lemma, it is enough to let $\mu \leq np$, and for the second statement of the lemma, it is enough to let $\mu \geq np$.
We use the following corollary of the Chernoff bound.

\begin{lemma}
\label{lem:simple_chernoff}
Let $X$ be a binomially distributed variable with parameters $n$ and $p$, and let $0 < \gamma < 1$. Then,
\[ \Pr(|X - pn| \geq \gamma n) \leq 2 \exp \left ( - \frac{1}{12} \gamma^2 n \right ). \]
\end{lemma}

\begin{proof}
First, suppose that $p \geq \frac{1}{2} \gamma$. Since $X$ is a binomially distributed variable with mean of $np$, we can apply Lemma \ref{lem:chernoff} with $\mu = np$ and $\delta = \gamma / 2 p \leq 1$.
Then, by Lemma \ref{lem:chernoff}, our inequality fails with probability at most 
$2 \exp \left (-\frac{1}{12} \frac{\gamma^2}{p} n \right ) \leq 2 \exp \left ( - \frac{1}{12} \gamma^2 n \right )$.

Next, suppose that $p < \frac{1}{2}\gamma$. In this case, we only need to prove the upper bound of the inequality. Since $p n < \gamma n$, we can apply Lemma \ref{lem:chernoff} with $\mu = \gamma n$ and $\delta = 1$ to show that the upper bound fails with probability at most $\exp (-\frac{1}{3} \gamma n) < 2 \exp \left ( - \frac{1}{12} \gamma^2 n \right )$.
\end{proof}



We also use the following well-known symmetric form of the Lov\'asz Local Lemma.

\begin{lemma}[\cite{shearer1985problem}]
\label{lem:LLL_symmetric}
Let $\mathcal A$ be a collection of (bad) events in a probability space. Suppose that each bad event in $\mathcal A$ occurs with probability at most $p$ and is independent with all but fewer than $D$ other bad events in $\mathcal A$. If 
\[Dp < 1/e,\]
then with positive probability, no bad event in $\mathcal A$ occurs.
\end{lemma}

Now, we are ready to prove Theorem \ref{thm:reg_intro}.
Writing $\alpha = \sqrt 2 - 1$, we can restate Theorem \ref{thm:reg_intro} as follows. 

\begin{theorem}
\label{thm:regular}
Let $G$ be a graph of minimum degree $\delta$ and maximum degree $\Delta = e^{o(\delta)}$, and 
let $F:V(G) \rightarrow 2^{\mathbb N}$.
If \[|F(v)| \leq \left (\alpha - o(1) \right ) \deg_G(v)\]
for each $v \in V(G)$, then $G$ has an $F$-avoiding orientation.
\end{theorem}

In the statement of the theorem, our function $o(1)$ approaches $0$ as $\delta \rightarrow \infty$, and the rate at which $o(1)$ approaches $0$ depends on how quickly $\delta$ dominates $\log \Delta$.

\begin{proof}
Let $\gamma > 0$ be a fixed value. We randomly choose an ordering of $V(G)$ by using a function $\phi$ to map each vertex $v \in V(G)$ uniformly at random to the real interval $[0,1]$ and then letting $V(G)$ be ordered according to the left-right order of the image of $V(G)$ in $[0,1]$. Since two vertices are mapped to the same value with probability $0$, we can assume that all values $\phi(v)$ are distinct.

Consider a vertex $v \in V(G)$, and write $d = \deg_G(v)$. We observe that by Lemma \ref{lem:simple_chernoff}, it holds with probability at least $1-2\exp\left (-\frac{1}{12} \gamma^2 d \right )$ that 
\begin{equation}
\label{eqn:g-}
(\phi(v) - \gamma) d < \deg^L_G(v) <( \phi(v)  + \gamma )d.
\end{equation}


Now, we probabilistically create a subgraph $H \subseteq G$. If $uv \in E(G)$ and $\phi(u) = x$ and $\phi(v) = y$, then we add the edge $uv$ to $H$ with probability 
\[ p(x,y) = 
\begin{cases}
\frac{\alpha}{1-\alpha} & \textrm{if } 0 \leq x \leq \alpha \textrm{ and } \alpha + \left ( \frac{1-\alpha}{\alpha}  \right )x \leq y \leq 1, \\
0 & \textrm{otherwise.}
\end{cases}
\]
We make some observations about $H$. 
First, by construction, a vertex $v \in V(G)$ for which $\phi(v) \leq \alpha$ satisfies $\deg_H^L(v) = 0$, and a vertex $v \in V(G)$ for which $\phi(v) \geq \alpha$ satisfies $\deg_H^R(v) = 0$.
Next, consider a vertex $v \in V(G)$, and write $d = \deg_G(v)$. If $\phi(v) \leq \alpha$, then we claim that by our simplified Chernoff bound (Lemma \ref{lem:simple_chernoff}), it holds with probability at least $1 - 2\exp \left ( - \frac{1}{12}\gamma^2 d \right )$ that 
\begin{equation}
\label{eqn:h+}
 (\alpha - \phi(v) - \gamma) d \leq \deg^R_H(v) \leq (\alpha - \phi(v)+ \gamma )d. 
\end{equation}
Indeed, each neighbor
$w \in N(v)$ becomes a right-neighbor of $v$ in $H$ independently with probability 
\[\int_0^1 p(\phi(v),y) dy = \frac{\alpha}{1-\alpha} \cdot (1 - \alpha - \frac{1- \alpha}{\alpha}\phi(v)) = \alpha - \phi(v),\]
and hence Lemma \ref{lem:simple_chernoff} applies.

On the other hand, if $\phi(v) \geq \alpha$, we claim that Lemma \ref{lem:simple_chernoff} implies that, with probability at least $1 - 2\exp \left ( - \frac{1}{12}\gamma^2 d \right)$, we have 
\begin{equation}
\label{eqn:h-}
\frac{1}{2} (\phi(v)-\alpha - \gamma) d \leq \deg^L_H(v) \leq \frac{1}{2} (\phi(v)-\alpha +  \gamma )d. 
\end{equation}
Indeed, each neighbor $w$ of $v$ becomes a left-neighbor of $v$ in $H$ with probability 
\[\int_{0}^{1} p(x,\phi(v)) dx =  \left ( \frac{\alpha}{1-\alpha} \right )^2 (\phi(v) - \alpha) = \frac{1}{2}(\phi(v) - \alpha),\]
so Lemma \ref{lem:simple_chernoff} applies.

Now, depending on the value of $\phi(v)$, we see from (\ref{eqn:g-}) and (\ref{eqn:h+}) or from (\ref{eqn:g-}) and (\ref{eqn:h-}) that the distribution on $\deg^L_G(v) - 2\deg^L_H(v) + \deg^R_H(v)$ is centered around $\alpha d$. In particular, with probability at least $1 - 4\exp(-\frac{1}{12} \gamma^2 d)$,
\begin{equation}
\label{eqn:b}
(\alpha   - 2\gamma )d < \deg^L_G(v) - 2\deg^L_H(v) + \deg^R_H(v) < (\alpha +  2\gamma) d . 
\end{equation}


Now we apply the Lov\'asz Local Lemma (Lemma \ref{lem:LLL_symmetric}). For each vertex $v$, we define a bad event $B_v$ to be the event that (\ref{eqn:b}) does not hold for $v$. We have seen that the probability of each bad event $B_v$ is at most $4\exp(- \frac{1}{12} \gamma^2 \delta) $. Furthermore, the vertices and edges involved in a given bad event $B_v$ induce a star in $G$ centered at $v$, and two bad events are independent if and only if their corresponding stars are disjoint. Hence, $B_v$ is dependent with at most $\Delta^2$ other bad events. Hence, by Lemma \ref{lem:LLL_symmetric}, as long as 
\[ (\Delta^2+1) \cdot 4 \exp \left(-\frac{1}{12} \gamma^2 \delta \right)  < 1/e, \]
the equation (\ref{eqn:b}) holds for each vertex $v$ with positive probability. Since $\Delta = \exp(o(\delta))$, the inequality holds when $\delta$ is sufficiently large, and we hence conclude that with positive probability, our random process chooses a subgraph $H \subseteq G$ that satisfies (\ref{eqn:b}) at each vertex $v \in V(G)$. Thus, a
subgraph $H$ 
satisfying (\ref{eqn:b})
exists. Therefore, by Theorem \ref{thm:h}, if each vertex $v \in V(G)$ has a list $F(v)$ of at most $(\alpha - 2 \gamma) \deg_G(v)$ forbidden values, then $G$ has an $F$-avoiding orientation.
Letting $\gamma$ tend to $0$ completes the proof.
\end{proof}

Finally, we show that if we use Theorem \ref{thm:h} to prove for some constant $\beta$ that every regular graph $G$ has an $F$-avoiding orientation when $|F(v)| \leq \beta \deg(v)$, then $\beta \leq \alpha+o(1)$. In other words, the coefficient $\alpha-o(1)$ in Theorem \ref{thm:regular}
is essentially best possible. To show this, consider the complete graph $G = K_n$, and order its vertices as $v_1, \dots, v_n$.
Suppose that we can use Theorem \ref{thm:h} to show that $G$ has an $F$-avoiding orientation whenever $|F(v)| \leq \beta(n-1)$ at each vertex $v \in V(G)$.
We partition $V(K_n)$ into two sets $A$ and $B$, where $A$ contains all vertices $v_i$ for which $i \leq \lfloor \beta n\rfloor$, and $B$ contains all other vertices.
Given an oriented subgraph $H$ of $G$, we say that the \emph{weight} $w(v)$ of a vertex $v \in V(G)$ is given by 
\[w(v) = \deg^L_G(v) - 2 \deg^L_H(v) + \deg_H^R(v).\]
In order to apply Theorem \ref{thm:h}, the weight of each vertex $v \in V(G)$ must satisfy $w(v) \geq \beta(n-1)$.

Now, let $e(A)$ and $e(B)$ denote the number of edges in $H[A]$ and $H[B]$, respectively, and let $e(A,B)$ denote the number of edges with one endpoint in $A$ and one endpoint in $B$. We observe that if Theorem \ref{thm:h} applies, then the total weight $w(A)$ of all vertices in $A$ satisfies
\begin{equation}
\lfloor \beta n \rfloor \beta (n-1) \leq  w(A) = \sum_{i = 1}^{\lfloor \beta n \rfloor } (i-1) - e(A) + e(A,B) \leq \frac{1}{2} \lfloor \beta n -1 \rfloor  \lfloor \beta n \rfloor + e(A,B), \label{eq:a}
\end{equation}
and the total weight $w(B)$ of all vertices in $B$ satisfies
\begin{equation}
 (n -  \lfloor \beta n \rfloor  ) \beta (n-1) \leq w(B) = \sum_{i = \lfloor \beta n \rfloor + 1}^{ n } (i-1)  -2e(A,B) - e(B) \leq \frac{1}{2} \left ( n^2 - n - \lfloor \beta n \rfloor^2 - \lfloor \beta n \rfloor \right ) -2e(A,B).
 \label{eq:b}
 \end{equation}
We eliminate the quantity $e(A,B)$ by doubling (\ref{eq:a}) and adding (\ref{eq:b}), and then simplifying yields 
\[(1+o(1))n^2 \left ( \frac{1}{2}\beta^2 + \beta - \frac{1}{2} \right ) \leq 0.\]
This implies that $\beta \leq \alpha + o(1)$. Therefore, the coefficient of $\alpha - o(1)$ appearing in Theorem \ref{thm:regular} is close to best possible when using Theorem \ref{thm:h}.

\section{Conclusion}\label{sec:conc}




It is natural to ask whether Conjecture \ref{conj:1} could be proved for every graph $G$ using Theorem \ref{thm:2} by finding a suitable orientation $D$ of $G$, which would give a stronger result. In other words, the following question emerges.

\begin{question}
\label{q:stronger}
Does every graph $G$ have a spanning subgraph with an Alon-Tarsi orientation 
in which each vertex $v$ has out-degree at least 
$\frac{1}{2}(\deg_G(v) - 1)$?
\end{question}

If Question \ref{q:stronger} has a negative answer, we can still ask the similar question where $-1$ is replaced with an arbitrary constant $-C$.

If the answer to Question \ref{q:stronger} is affirmative, then it would imply that every graph $G$ on $n$ vertices and of maximum degree $\Delta$ has a subgraph $H$ with at least $|E(G)| -  n/2  $ edges for which $\AT(H) \leq \lfloor \frac{1}{2} \Delta + \frac{3}{2} \rfloor$. Indeed, given an affirmative answer to Question \ref{q:stronger} and a graph $G$, we could first find a subgraph $H$ of $G$ with an Alon-Tarsi orientation $D$ as described in the question. This subgraph necessarily has at least $\sum_{v \in V(G)} \frac{1}{2}(\deg_G(v) - 1) = |E(G)| - n/2$ edges. Then, after reversing the direction given by $D$ for each edge of $H$, we obtain an Alon-Tarsi orientation of $H$ in which each vertex has out-degree at most $\frac{1}{2}\deg_G(v) + \frac{1}{2}$, implying that $\AT(H) \leq \lfloor \frac{\Delta + 3}{2} \rfloor$. This leads us to the following problem.

\begin{question}
\label{q:matching}
Does there exist a constant $C$ such that every graph $G$ of maximum degree $\Delta$ has a matching $M$ such that $\AT(G - M) \leq  \frac{1}{2}\Delta + C$?
\end{question}

If the answer to Question \ref{q:matching} is affirmative, it is natural to ask for the optimal value of $C$. We note that $C$ cannot be smaller than $\frac{1}{2}$, since for any odd $n$ and any maximum matching $M$ in $K_n$, $\AT(K_n - M) = \frac{n+1}{2}$ (see \cite{Huang}).

Our requirement that $M$ be a matching rather than an arbitrary set of $n/2$ edges is inspired by the following two results, which both show that if a matching is removed from a graph $G$, then the upper bound on some coloring parameter of $G$ can be reduced. For planar graphs, Zhu \cite{ZhuPlanar} has shown that the Alon-Tarsi number is at most $5$ and that this upper bound is tight, but Grytczuk and Zhu \cite{GZ} have also shown that every planar graph $G$ contains a matching $M$ for which $\AT(G-M) \leq 4$. Furthermore, it is well-known that a graph $G$ of maximum degree $\Delta$ satisfies $\chi(G) \leq \Delta  + 1$ and that this bound is tight for cliques.
However, every graph
$G$ of maximum degree $\Delta$ contains a matching $M$ for which $\chi(G-M) \leq \lceil \frac{\Delta+1}{2} \rceil$.
Indeed, Lov\'asz \cite{LovaszDefective}
proved a result on graph coloring which implies, as shown by Jesurum \cite{Jesurum}, that 
every graph of maximum degree $\Delta$ has a
$\lfloor \Delta /k \rfloor$-defective
$k$-coloring, in which each vertex is adjacent to at most
$\lfloor \Delta /k \rfloor$ vertices of the same color.
By letting $k = \lceil \frac{\Delta+1}{2} \rceil$, 
we see that every graph has a $1$-defective $\lceil \frac{\Delta+1}{2} \rceil$-coloring, in which the monochromatically colored edges form a matching. By removing this matching $M$ of monochromatic edges from $G$, we obtain a graph 
with chromatic number at most $\lceil \frac{\Delta+1}{2} \rceil$.

\bibliographystyle{abbrv}
\bibliography{ref}

\end{document}